\documentclass[11pt,a4paper,english]{article} 
\usepackage{import}
\usepackage[utf8]{inputenc}
\usepackage{babel}                
\usepackage{float}   
\usepackage{url}    
\usepackage{sectsty}               
\usepackage{tabularx}              
\usepackage{titling}              
\usepackage{imakeidx}              
\usepackage{xcolor}                
\usepackage{enumitem}              
\usepackage{amsmath}                                   
\usepackage{amssymb}                                    
\usepackage{amsthm}
\usepackage[Gray,squaren,thinqspace,thinspace]{SIunits} 
\usepackage{listings}                                   
\usepackage[small,bf,hang]{caption}        
\usepackage{subcaption}                    
\usepackage{sidecap}                       
\usepackage[colorlinks=true,citecolor=red,linkcolor=blue]{hyperref} 
\usepackage[capitalise,noabbrev]{cleveref}
\usepackage{url}                           
\usepackage{cite}                          
\usepackage{tikz}          
\usepackage{pgfplots}      
\usepackage{pgfplotstable} 
\usepackage{placeins}      
\usepackage{graphicx}      
\usepackage{longtable}    
\usepackage{mathrsfs}
\usepackage{authblk}
\pgfplotsset{compat=1.18} 
\usepackage{tikz-cd}
\usetikzlibrary{positioning}
\usetikzlibrary{decorations.pathmorphing}
\usetikzlibrary{knots}
\usetikzlibrary{calc}
\theoremstyle{plain}
\newtheorem{thm}{Theorem}[section]
\newtheorem{prop}[thm]{Proposition}
\newtheorem{lem}[thm]{Lemma}

\newtheorem*{thm*}{Theorem}

\theoremstyle{definition} 

\newtheorem{defn}[thm]{Definition}
\newtheorem*{rmk}{Remark}

\Crefname{thm}{Theorem}{Theorems}
\Crefname{prop}{Proposition}{Propositions}
\Crefname{lem}{Lemma}{Lemmas}
\Crefname{cor}{Corollary}{Corollaries}
\Crefname{defn}{Definition}{Definitions}
\Crefname{tab}{Table}{Tables}
\Crefname{ex}{Example}{Examples}
\Crefname{chap}{Champter}{Chapters}
\Crefname{app}{Appendix}{Appendices}
\usepackage{stmaryrd}

\newcommand{\ra}{\rightarrow}

\newcommand{\ot}{\otimes}

\newcommand{\homgrp}[2]{\mathrm{Hom}(#1,#2)}

\newcommand{\Hom}{\mathrm{Hom}}
\newcommand{\ignore}[1]{}
\newcommand{\op}{\mathrm{op}}

\newcommand{\cee}{\mathcal{C}}
\newcommand{\stau}{*-autonomous}
\newcommand{\eps}{\varepsilon}
\newcommand{\gv}{Grothendieck-Verdier}

\title{\textbf{
Hopf algebroids and Grothendieck-Verdier duality
}}
\author{Robert Allen}
\affil{\small{School of Mathematics, University of Bristol, Bristol, BS8 1TW, UK, \\ Heilbronn Institute for Mathematical Research, Bristol, UK.}}

\begin{document}

\maketitle

\begin{abstract}
    Grothendieck-Verdier duality is a powerful and ubiquitous structure on monoidal categories, which generalises the notion of rigidity. Hopf algebroids are a generalisation of Hopf algebras, to a non-commutative base ring. Just as the category of finite-dimensional modules over a Hopf algebra inherits rigidity from the category of vector spaces, we show that the category of finite-dimensional modules over a Hopf algebroid with bijective antipode inherits a Grothendieck-Verdier structure from the category of bimodules over its base algebra. We investigate the algebraic and categorical structure of this duality.
\end{abstract}

\section{Introduction}

Hopf algebras are defined precisely in such a way that the category of their finite-dimensional modules inherits rigidity from the category of vector spaces. The rigid dual module structure comes from the antipode map. Hopf algebroids are generalisations of Hopf algebras, defined in the category of bimodules over some algebra $A$ over a field $k$. While the category of $A$-$A$-bimodules is not rigid in general, it is natural to ask if the category of Hopf algebroid
modules that are finite-dimensional as $k$-vector spaces
inherits some duality structure from this underlying category. We show that this is correct, and the duality structure in question is \gv{}, also called non-symmetric \stau{}, which is a generalisation of rigidity. The crucial step in this process is to determine the dualising object, for which the vector space dual of the base algebra is a natural choice. We show that the antipode map equips this space with an action of the Hopf algebroid, and that the corresponding dualising functor is an anti-equivalence, hence endowing the category of modules with a \gv{} structure.

\subsection{Hopf algebroids}

Hopf algebroids are generalisations of Hopf algebras, with a noncommutative base algebra $A$. Recently, they have found applications in quantum gravity \cite{LMM+18}, noncommutative geometry \cite{Gho20}, differential equations \cite{KG21} and Hopf-Galois extensions \cite{HM23}, to name a few.
While there is an accepted definition of bialgebroid, there are different definitions of Hopf algebroid, and we work with the definition from \cite[Definition 4.1]{BS04}, which has been shown to be equivalent to the definition from \cite[Section 11]{DS04}, by \cite{BS04}[Theorem 4.7]. We call this a full Hopf algebroid, or sometimes just a Hopf algebroid (with bijective antipode), which is a special case of a left Hopf algebroid, or $\times_A$-Hopf algebra, as defined in \cite{Sch00}. In general, the category of finite-dimensional modules over a Hopf algebroid is not rigid. If a rigid category is desired, then it is typical to consider the category of modules which are finitely generated projective as left (or right) modules over the base algebra, when each module $M$ has a rigid dual given by $\Hom_A(M, A)$ 
\cite[Proposition 4.41]{Boh09}.

\subsection{Grothendieck-Verdier categories}

\gv{} categories are closed monoidal categories with a duality structure that generalises rigidity \cite{Bar95, BD13}. They are also known as non-symmetric \stau{}, or just \stau{}, and we will use these terms interchangeably. Many of the results associated with rigid categories have analogues in the \gv{} setting. For example, the Eilenberg-Watts theorem \cite[Lemma 3.7]{FSSW23}. \gv{} categories arise in a broad range of mathematical fields \cite{HS03}. For example, they have found applications in logic \cite{M09}, quantum theory \cite{Day11}, algebra \cite{Man17}, conformal field theory \cite{ALSW21} and algebraic topology \cite{MW22} and finite tensor categories \cite{MW23}.

\subsection{Known results}

The connections between Hopf algebroids and \gv{} categories have been noticed before. In \cite{DS04}[Example 7.4], Hopf algebroids are characterised by a ``strong \stau{} structure'' on an opmorphism between pseudomonoids in the bicategory of algebras, bimodules and bimodule homomorphisms. 
This gives rise to a strong monoidal and strong closed functor from the category of Hopf algebroid modules to the category of bimodules over the base algebra \cite{DS04}[Section 11]. However, it is not clear that this is sufficient to lift the duality structure, and further this definition of Hopf algebroid, although equivalent, is presented quite differently to the more common algebraic definition and so would provide little insight on the explicit structure of the duality.

It is also known that a \stau{} monad lifts the \stau{} structure to its category of algebras \cite{Pas12, PS09}. 
For a Hopf algebroid, the underlying category is that of bimodules over a finite-dimensional algebra, which is \stau{} \cite[Remark 3.17]{FSS19}, \cite[(2.22)]{FSSW23}. It remains to show that Hopf algebroids are in fact \stau{} monads. Hopf algebroids are known to be examples of Hopf monads \cite{BLV11}, and a sufficient condition for a Hopf monad to be a \stau{} monad is that the dualising object has an algebra structure \cite{HL18}. This can be established for a Hopf algebroid with invertible antipode, as we will show. 
However, the main result of \cite{HL18} relies upon a theorem which claims that ``the notions of linear distributive categories with negation and \stau{} categories coincide", for which the references \cite{CS97, Pas12} are provided. \cite[Theorem 4.5]{CS97} claims to prove this result in the symmetric case, but a large part of the proof is left ``to the faith of the reader". In fact, obtaining the distributors turns out to be surprisingly subtle, see \cite{FSSW23}. 

 Inspired by these connections, 
 we believe it is valuable to provide an explicit and self-contained proof, which leaves no room for doubt.

\subsection{Outline}

In \cref{sec:hopfalg}, we introduce the definitions of bialgebroid and Hopf algebroid (with bijective antipode) and collect some useful consequences and remarks.

In \cref{sec:closed}, we exhibit the closed structure of the category of modules over a Hopf algebroid, and express the internal Hom functors in terms of the internal Hom functors of the base algebra. We write down the action of the Hopf algebroid on these internal Hom spaces.

In \cref{sec:dual} we define a Hopf algebroid module structure on the dual of the base algebra and relate the vector space dual of a module to the internal Hom functors to this module. We state a definition of \gv{} category, and finally we prove that the category of finite-dimensional modules over a Hopf algebroid with bijective antipode is a \gv{} category, where a natural choice of dualising object is the dual of the base algebra.  

Usually it will be clear from the context in which set an element takes its value. We tend to use $H$ for the Hopf algebroid  (total algebra) and $A$ for the base algebra, $M$ for an $H$-module and $M^*$ for its $k$-linear dual. Then the elements are usually denoted $h, k \in H$, $a,b,c \in A$, $m \in M$ and $f \in M^*$. 

 \section{Hopf algebroids}\label{sec:hopfalg}

In this section, we define the notions of bialgebroid and Hopf algebroid. We also collect some useful consequences and remarks. Let $k$ be a field and let $A$ be a finite-dimensional $k$-algebra.

\begin{defn}[{\cite[Definition 2.1]{BS04}}]
    Let $H$ be a $k$-algebra. Then $H$ is a \emph{bialgebroid} over the \emph{base algebra} $A$ if it is equipped with
    \begin{itemize}
        \item two $k$-algebra maps $\alpha : A \ra H$, $\beta: A^\op \ra H$ satisfying 
        \begin{equation}
            \alpha(a)\beta(b) = \beta(b) \alpha(a).
        \end{equation}
        These endow $H$ with the $A$-bimodule structure 
        \begin{equation}
            a \cdot h \cdot b = \alpha(a) \beta(b) h.
        \end{equation}
        \item a coassociative comultiplication $\Delta : H \ra H \ot_A H$, where $\ot_A$ is the tensor product of $A$-bimodules, and a counit $\eps : H \ra A$, which are unital maps which satisfy
        \begin{equation}
            \begin{aligned}
            &\Delta(\alpha(a)\beta(b) h) = \alpha(a) h_1 \ot_A \beta(b) h_2 ,  \\ 
        &\Delta(h h') = \Delta(h)\Delta(h'),  \\ 
        & h_1 \beta(a) \ot_A h_2 = h_1 \ot_A h_2 \alpha(a),  \\
        &\eps(\alpha(a) \beta(b) h ) =  a \eps(h) b ,  \\
        &\eps(h h') = \eps(h \alpha(\eps(h'))) = \eps( h \beta(\eps(h'))),  \\
        &\alpha(\eps(h_1))h_2 = \eps(h_1) \cdot h_2 =  h = h_1 \cdot \eps(h_2) = \beta(\eps(h_2))h_1.
    \end{aligned}
    \end{equation}
    where we have used Sweedler notation to write $\Delta(h) = h_1 \ot_A h_2$.
    \end{itemize}   
\end{defn}

\begin{defn}[{\cite[Definition 4.1]{BS04}}]
Let $H$ be a bialgebroid with base algebra $A$. Then $H$ is a \emph{Hopf algebroid} with base algebra $A$ if it is equipped with an invertible anti-automorphism $S: H \ra H$, called the \emph{antipode}, satisfying 
$S \circ \beta = \alpha$, and
\begin{equation}\label{eq:antipode}
     \begin{aligned}
         S(h_1)_1 h_2 \ot_A S(h_1)_2 &= 1_H \ot_A S(h),  \\
         (S^{-1} h_2)_1 \ot_A (S^{-1} h_2)_2 h_1 &= S^{-1} (h) \ot_A 1_H .
\end{aligned}
\end{equation}
\end{defn}

\begin{rmk}
    The definition of a full Hopf algebroid, or Hopf algebroid with invertible antipode has many equivalent formulations \cite[Proposition 4.2]{BS04}, one of which is a pair of a left and right bialgebroid, such that the antipode maps between them. We choose the definition above as it will be more productive in computing the action on certain modules later.
\end{rmk}


\begin{rmk}
\begin{itemize}
    \item A bialgebroid $H$ acts on its base algebra
    $A$ by 
    \begin{equation}
        h \cdot a = \eps(h \alpha(a)) = \eps(h \beta(a)).
    \end{equation}
    \item Every $H$-module is an $A$-$A$-bimodule with the action
    \begin{equation}
        a \cdot m \cdot b = (\alpha(a) \beta(b)) \cdot m.
    \end{equation}
    \item The maps $\alpha$ and $\beta$ also endow $H$ with a second $A$-bimodule structure
    \begin{equation}
        a \cdot h \cdot b = h \alpha(a)\beta(b) .
    \end{equation}
\end{itemize}
\end{rmk}

\section{Closed structure} \label{sec:closed}
In this section, we prove that the left/right internal Hom functors for modules over the Hopf algebroid can be expressed in terms of the internal Homs of left/right modules over the base algebra. This is already known - see \cite{Sch00}. We write down the isomorphisms explicitly in order to present the module structure of the internal Hom, for later use.

\begin{defn}
    Let $\cee$ be a monoidal category. $\cee$ is called left/right \emph{closed} if it possesses a left/right internal Hom, which is a functor 
    \begin{equation}
        [-,-]^{r/l} : \cee^\op \times \cee \ra \cee,
    \end{equation}
    such that there are the following natural isomorphisms:
    \begin{equation}
         \homgrp{X}{[Y,Z]^r} \cong \homgrp{X \ot Y}{Z} \cong \homgrp{Y}{[X,Z]^l}.
    \end{equation}
\end{defn}

\begin{prop}
    Let $M$, $N$ be $H$-modules. The category of $H$-modules is right and left closed, with
    \begin{equation}
    \begin{aligned}
        [M, N]^r &= \Hom_H(H \ot_A M, N) , \\
        [M, N]^l &= \Hom_H(M \ot_A H, N) ,
    \end{aligned}
    \end{equation}
    with left action of $H$ given by right multiplication on $H$ in $H \ot_A M$ and $M \ot_A H$, where the tensor products are in the category of $H$-modules.
\end{prop}
\begin{proof}
    Our definition of Hopf algebroid is in particular a left Hopf algebroid \cite[Proposition 4.2]{BS04}. Left Hopf algebroids are known to be closed, with this internal Hom functor - for the right closed case, see \cite[Proposition 3.3]{Sch00}, see also \cite[Section 4.6.2]{Boh09}.
\end{proof}

\begin{lem}\label{thm:fusion}
Let $M$ be a left $H$-module. Then the module $H \ot_A M$, with the 
following action of $H$
\begin{equation}
    h \cdot (h' \ot_A m) =   h_1 h' \ot_A h_2 \cdot m,
\end{equation}
is isomorphic to $H\ot_{A^\op} M$, with left multiplication on $H$
\begin{equation}
    h \cdot (h' \ot_{A^\op} m) = h h' \ot_{A^\op} m .
\end{equation}
Here, the symbol $\ot_{A^\op}$ means the bimodule tensor product over the right action of $A^{\op}$ on $H$ given by $h \triangleleft a = h \beta(a)$ and the left action of $A^{\op}$ on $M$ given by $a \triangleright m = \beta(a) \cdot m$. That is, we have
\begin{equation}
      h \beta(a) \ot_{A^\op} m = h \ot_{A^\op} \beta(a) \cdot m = h \ot_{A^\op}  m \cdot a.
\end{equation}

Similarly $M \ot_A H$ and $M \ot_{A^\op} H$, equipped with the following actions, are isomorphic.
\begin{equation}
    \begin{aligned}
    h \cdot (m \ot_A h') &=   h_1 \cdot m \ot_A h_2 h', \\
    h \cdot (m \ot_{A^\op} h') &= m \ot_{A^\op} hh' ,
\end{aligned}
\end{equation}
where this bimodule tensor product is over the right action of $A^\op$ on $M$ given by $m \blacktriangleleft a = \alpha(a) \cdot m $ and the left action of $A^\op$ on $H$ given by $a \blacktriangleright h = h \alpha(a)$, so
\begin{equation}
    m \ot_{A^\op} \alpha(a) h = \alpha(a) \cdot m \ot_{A^\op} h = a \cdot  m \ot_{A^\op}  h.
\end{equation}
\end{lem}
\begin{proof}
The first isomorphism is given by
\begin{equation}
    \begin{aligned}
H \ot_{A^\op} M & \cong H \ot_{A} M  \\
    h \ot_{A^\op} m &\mapsto h_1 \ot_{A} h_2 \cdot m  \\
    S^{-1}(S(h)_2) \ot_{A^\op} S(h)_1 \cdot m &\mapsfrom h \ot_{A} m,
\end{aligned}
\end{equation}

We can see that this is an $H$-module homomorphism by
\begin{equation}
    (hh')_1 \ot_{A} (hh')_2 \cdot m = h_1 h'_1 \ot_A h_2 h'_2 \cdot m .
\end{equation}
Similarly, for the second isomorphism we have
\begin{equation}
    \begin{aligned}
M \ot_{A^\op} H & \cong M \ot_{A} H  \\
    m \ot_{A^\op} h &\mapsto h_1  \cdot m \ot_{A} h_2  \\
    S^{-1}(h)_2  \cdot m \ot_{A^\op} S(S^{-1}(h)_1) &\mapsfrom m \ot_{A} h, 
\end{aligned}
\end{equation}
Note that these isomorphisms are essentially the same as those denoted by $\alpha$ and $\beta$ in {\cite[Proposition 4.2]{BS04}}.
The tensor products over $A^\op$ are required for these isomoprhisms to be well defined. For example
\begin{equation}
    h \beta(a) \ot_{A^\op} m \mapsto h_1 \ot_A (h_2 \beta(a)) \cdot m = h_1 \ot_A h_2 \cdot (\beta(a) \cdot m) \mapsfrom h \ot_{A^\op} \beta(a) \cdot m.
\end{equation}
\end{proof}


\begin{prop}\label{thm:actions}
The following are isomorphic, as $H$-modules:
\begin{equation}
    \begin{aligned}
\Hom_H(H \ot_A M, N) &\cong \Hom_H(H \ot_{A^\op} M, N) \cong \Hom_{A^{\op}}(M, N), \\
\Hom_H(M \ot_A H, N) &\cong \Hom_H(M \ot_{A^\op} H, N) \cong \Hom_A(M, N),
\end{aligned}
\end{equation}
where the $H$-actions on $\Hom_{A^{\op}}(M,N)$ and $\Hom_A(M,N)$ are given by
\begin{equation}
    \begin{aligned}
    (h \cdot f)(m) &= S^{-1}(S(h)_2) \cdot f( S(h)_1 \cdot m),  \\
    (h \cdot f)(m) &= S(S^{-1}(h)_1) \cdot f( S^{-1}(h)_2 \cdot m),
\end{aligned}
\end{equation}
respectively. Here, a subscript $A$/$A^{\op}$ on the hom space denotes left/right $A$-module homomorphisms, respectively. 
\end{prop}
\begin{proof}
For the right closed case, see \cite[Theorem and Definition 3.5]{Sch00}. 
The first isomorphisms of each line follow from \cref{thm:fusion}, and transforms the diagonal action on $H \ot_A M$ to the action on $H$ only. 
The second isomorphisms are given by
\begin{equation}
    \begin{aligned}
\Hom_H(H \ot_{A^\op} M, N) &\cong \Hom_{A^{\op}}(M, N),  \\
f &\mapsto [m \mapsto f(1_H \ot_{A^\op} m)],  \\
[h \ot_{A^\op} m \mapsto h \cdot g(m)] & \mapsfrom g  .
\end{aligned}
\end{equation}
We check that both directions give the appropriate types of homomorphisms.
\begin{equation}
    \begin{aligned}
    &g(m) \cdot a = \beta(a) \cdot f(1_H \ot m) = f(\beta(a) \ot m)  = g(m \cdot a),  \\
   & h \cdot f(h' \ot m) = h h' \cdot g(m) = f(hh' \ot m) = f(h \cdot (h' \ot m)) .
\end{aligned}
\end{equation}
Similarly, we have
\begin{equation}
\begin{aligned}
\Hom_H(M \ot_{A^\op} H, N) &\cong \Hom_A(M, N),  \\
f &\mapsto [m \mapsto f(m \ot_{A^\op} 1_H )],  \\
[m \ot_{A^\op} h \mapsto h \cdot g(m)] & \mapsfrom g  .
\end{aligned}
\end{equation}
Again we check
\begin{equation}
    \begin{aligned}
    &a \cdot g(m) =  \alpha(a) \cdot f(m \ot 1_H) = f(m \ot \alpha(a)) = f( a \cdot m \ot 1_H) = g( a \cdot m),  \\
    &h \cdot f(m \ot h') = h h' \cdot g(m) = f(m \ot hh') = f(h \cdot (m \ot h')).
\end{aligned}
\end{equation}
The $H$-actions are calculated by applying the isomorphisms to the actions on the left-hand modules.
\end{proof}

\section{Duality structure} \label{sec:dual}

In this section, we prove that the internal Hom from a given module to the dual of the base algebra is isomorphic to the dual vector space of the original module, with action given by the antipode. Finally, we use this fact to prove that the category of finite-dimensional modules over a Hopf algebroid with bijective antipode is \gv{}.

\begin{prop}
    The vector space $A^* = \Hom_k(A, k)$ can be equipped with the following module structures, for $a,b,c \in A$, $h \in H$, $f \in A^*$:
    \begin{itemize}
        \item $A$-bimodule with action 
        \begin{equation}
        (c \cdot f \cdot a)(b) = f(abc),
        \end{equation}
        \item $H$-module with action 
        \begin{equation}\label{eq:hact}
            (h \cdot f)(a) = f(S(h) \cdot a) = f(S^{-1}(h) \cdot a).
        \end{equation}
    \end{itemize}
\end{prop}
\begin{proof}
The only nontrivial claim in this proposition is the equality of the actions with $S$ and $S^{-1}$. In \cite[Section 2.6.8]{Kow09}, two different right bialgebroid structures are presented, which, together with the left bialgebroid and antipode, produce isomorphic full Hopf algebroids. By \cite[Proposition 4.3]{BS04}, these differ by an isomorphism which is trivial on the algebra $H$. We can deduce that there is an isomorphism $\phi : A \ra A$ such that
    \begin{equation}
        S \alpha = \beta \phi, \quad \phi \eps S^{-1} = \eps S \ \implies \ S^2 \beta = \beta \phi, \quad \eps S^2 = \phi \eps.
    \end{equation}
    This isomorphism yields an isomorphism of modules, between $A$ with the action of $h$, and $A$ with the action of $S^2(h)$.
    \begin{equation}
        \begin{aligned}
        \phi(h \cdot a) &= \phi(\eps(h \beta(a))) = \eps(S^2(h \beta(a))) = \eps(S^2(h) S^2(\beta(a))) \\ &= \eps(S^2(h) \beta(\phi(a))) =  h \cdot \phi(a).
    \end{aligned}
    \end{equation}
\end{proof}

Next we prove that the $k$-linear dual $M^*$ of any Hopf algebroid module $M$ inherits an action of $H$, from its realisation as both a left and right internal hom in the category of $H$-modules.
\begin{lem} 
The following are isomorphic, as vector spaces
\begin{equation}
    \begin{aligned}
M^* = \Hom_k(M, k) \cong \Hom_{A^\op}(M, A^*)  \\
f \mapsto [\phi_f(m) : a \mapsto f(m \cdot a)],
\end{aligned}
\end{equation}
Further, inducing the first action from \cref{thm:actions}, with $N = A^*$,
yields the following action of $H$ on $M^*$:
\begin{equation}
    (h \cdot f)(m) = f(S(h) \cdot m).
\end{equation}
\end{lem}
\begin{proof}
The isomorphism follows from the tensor-Hom adjunction
\begin{equation}
   \Hom_k(A \ot_{A^\op} M, k) \cong \Hom_{A^\op}(M, \Hom_k(A,k)).
\end{equation}
We check that the forward direction yields an $A$-module homomorphism.
 \begin{equation}
    \phi_{f}(m \cdot b)(a)  = f(m \cdot ba) = \phi_f(m)(ba) = (\phi_f (m) \cdot b)(a).
 \end{equation}
For the reverse direction, one can take $\phi \mapsto [m \mapsto \phi(m)(1) \in k]$. 
Then the action of $H$ on $M^*$ is defined by
\begin{equation}
    \begin{aligned}
(h \cdot f)(m) 
&= S^{-1}(S(h)_2) \cdot \phi_f( S(h)_1 \cdot m)(1) = \phi_f( S(h)_1 \cdot m)(S(h)_2 \cdot 1)  \\
&=   f((S(h)_1 \cdot m)) \cdot ( S(h)_2 \cdot 1) )  \\  
    &= f( (S(h)_1 \cdot m) \cdot \eps(S(h)_2)) = f(\beta(\eps(S(h)_2)) S(h)_1 \cdot m)  \\  
    &= f(S(h) \cdot m).
\end{aligned}
\end{equation}
\end{proof}

\begin{lem}
The following are isomorphic, as vector spaces
\begin{equation}
\begin{aligned}
M^* = \Hom_k(M, k) \cong \Hom_A(M, A^*)  \\
f \mapsto [\phi_f(m) : a \mapsto f(a \cdot m)].
\end{aligned}
\end{equation}
Further, inducing the second action from \cref{thm:actions}, with $N=A^*$,
yields the following action of $H$ on $M^*$:
\begin{equation}
    (h \cdot f)(m) = f(S^{-1}(h) \cdot m).
\end{equation}
\end{lem}
\begin{proof}
Similarly to the first case, we have
\begin{equation}
   \Hom_k(A \ot_{A} M, k) \cong \Hom_{A}(M, \Hom_k(A,k)).
\end{equation}
We check that
\begin{equation}
    \phi_{f}(b \cdot m)(a)  = f(ab \cdot m) = \phi_f(m)(ab) = (b \cdot \phi_f (m))(a).
\end{equation}
Again, similarly, we have
\begin{equation}
     \begin{aligned}
(h \cdot f)(m) &=  S(S^{-1}(h)_1) \cdot \phi_f( S^{-1}(h)_2 \cdot m)(1)  \\ &=    f((S^{-1}(h)_1 \cdot 1)) \cdot ( S^{-1}(h)_2 \cdot m) )  \\  &= f(  \eps(S^{-1}(h)_1) \cdot(S^{-1}(h)_2 \cdot m) ) = f(\alpha(\eps(S^{-1}(h)_1)) S^{-1}(h)_2 \cdot m)  \\  &= f(S^{-1}(h) \cdot m).
\end{aligned}
\end{equation}
\end{proof}

\begin{defn}{\cite{Bar95, BD13}}
    A \emph{Grothendieck-Verdier} category is a pair $(\cee, K)$, where $\cee$ is a monoidal category and $K \in \cee$, such that there is a natural isomorphism
    \begin{equation}
        \Hom(X \ot Y, K) \cong \Hom(X, D(Y)),
    \end{equation}
    where the contravariant functor $D$ is an anti-equivalence. $K$ is called a dualising object and $D$ a dualising functor.
\end{defn}

\begin{rmk}
    A closed monoidal category $\cee$ is \gv{} if there exists an object $K \in \cee$, such that the internal Hom functor $D = [-,K]^r$ is an anti-equivalence. 
\end{rmk}

We are now in a position to prove the main result.

\begin{thm} \label{thm:main}
    Let $H$ be a Hopf algebroid with finite-dimensional base $k$-algebra $A$ and an invertible antipode $S$.
    Then 
    the category 
    of finite-dimensional $H$-modules 
    is a \gv{} 
    category. A natural choice of dualising object is given by the vector space dual of the base algebra $A^*$, with the following action of $H$:
    \begin{equation}
        (h \cdot f)(a) = f(S(h) \cdot a), \qquad h \in H, \ f \in A^*, \ a \in A.
    \end{equation}
\end{thm}
\begin{proof}
    The dualising functor is defined by 
    \begin{equation}
        D = [-,A^*]^r \simeq \Hom_H(H \ot_A (-), A^*) \simeq \Hom_{A^\op}(-, A^*) \simeq \Hom_k (-, k),
    \end{equation}
    with the following $H$ action on $D(M)$, for some $H$-module $M$:
    \begin{equation}
        (h \cdot f)(m) = f(S(h) \cdot m).
    \end{equation}
    The inverse of the dualising functor is defined by
    \begin{equation}
        D^{-1} = [-,A^*]^l \simeq \Hom_H((-) \ot_A H, A^*) \simeq \Hom_A(-, A^*) \simeq \Hom_k (-, k),
    \end{equation}
    with the following $H$-action on $D^{-1}(M)$:
    \begin{equation}
        (h \cdot f)(m) = f(S^{-1}(h) \cdot m).
    \end{equation}
    As we are considering only finite-dimensional vector spaces, we have
    \begin{equation}
        D(D^{-1}(M)) \cong D^{-1}(D(M)) \cong \Hom_k(\Hom_k(M,k), k) \cong M.
    \end{equation}
    In the case $D(D^{-1}(M))$, the action on $\phi_m \in M^{**}$, defined by $\phi_m(f) = f(m)$ is given by
    \begin{equation}
        \begin{aligned}
        (h \cdot \phi_m)(f) &= \phi_m (S(h) \cdot f) = (S(h) \cdot f)(m) \\ &= f(S^{-1} (S (h)) \cdot m) = f(h \cdot m) = \phi_{h \cdot m}(f),
    \end{aligned}
    \end{equation}
    and similarly for $D^{-1} (D(M))$.
    Therefore the action on $M$ is the original one, and $D$ is an anti-equivalence.
\end{proof}

\begin{rmk}
    Hopf algebroids also satisfy the following relation.
    \begin{equation}
        S(S^{-1}(h)_2) \ot S(S^{-1}(h)_1)) = S^{-1}(S(h)_2) \ot S^{-1}( S(h)_1).
    \end{equation}
    This defines a second comultiplication on the Hopf algebroid, which parallels the definition of the second tensor product in a \gv{} category, given by
    \begin{equation}
        X \odot Y = D(D^{-1}Y \ot D^{-1}X) \cong  D^{-1}(DY \ot DX).
    \end{equation}
    In a rigid category these two tensor products are equivalent, just as one can see that for a Hopf algebra, the terms in the first line above coincide with the usual comultiplication.
\end{rmk}

\subsection{Acknowledgements}

Thanks to Tony Zorman for long discussions and proofreading. Thanks to Christoph Schweigert, Simon Wood and everyone at the conference ``Hopf algebroids and non-commutative geometry" at Queen Mary's university, for very helpful discussion and feedback. This
work was supported by the Additional Funding Programme for Mathematical
Sciences, delivered by EPSRC (EP/V521917/1) and the Heilbronn Institute
for Mathematical Research.

\end{document}